\documentclass[a4paper, 12pt]{amsart}

\usepackage[english]{babel}

\usepackage[latin1]{inputenc}
\usepackage[T1]{fontenc}
\usepackage{lmodern}

\usepackage{amsmath}
\usepackage{amssymb}
\usepackage{amsthm}

\usepackage{hyperref}

\usepackage{enumerate}

\renewcommand{\phi}{\varphi}
\renewcommand{\epsilon}{\varepsilon}
\renewcommand{\theta}[0]{\vartheta}

\newcommand{\exteriorpower}[0]{\bigwedge\nolimits\!}

\newcommand{\N}{\mathbb{N}}

\newcommand{\Size}[1]{\left\lvert #1 \right\rvert}

\newcommand{\Span}[1]{\left\langle\, #1 \,\right\rangle}

\newcommand{\Set}[1]{\left\{ #1 \right\}}

\newcommand{\Gp}[0]{\mathcal{G}(p)}
\newcommand{\Hp}[0]{H_{p}}

\newcommand{\Hc}[0]{\mathcal{H}}
\newcommand{\Jc}[0]{\mathcal{J}}

\newcommand{\norm}[0]{\trianglelefteq}

\newcommand{\F}{\mathbf{F}}

\DeclareMathOperator{\GL}{GL}

\DeclareMathOperator{\AGL}{AGL}

\DeclareMathOperator{\Aut}{Aut}

\DeclareMathOperator{\Frat}{Frat}
\DeclareMathOperator{\Hol}{Hol}

\DeclareMathOperator{\NHol}{NHol}

\DeclareMathOperator{\inv}{inv}

\newtheorem{dummy}{Dummy}
\numberwithin{dummy}{section}
\numberwithin{figure}{section}

\newtheorem{theorem}[dummy]{Theorem}

\newtheorem{remark}[dummy]{Remark}

\newtheorem{lemma}[dummy]{Lemma}

\newtheorem{proposition}[dummy]{Proposition}

\theoremstyle{definition}
\newtheorem{definition}[dummy]{Definition}

\theoremstyle{remark}

\numberwithin{equation}{section}

\begin{document}

\date{30 September 2019, 15:35 CEST --- Version 5.07%
}

\title[Multiple Holomorphs of $p$-groups of class two]%
      {Multiple Holomorphs of\\
      Finite $p$-Groups of Class Two}
      
\author{A. Caranti}

\address[A.~Caranti]%
 {Dipartimento di Matematica\\
  Universit\`a degli Studi di Trento\\
  via Sommarive 14\\
  I-38123 Trento\\
  Italy} 

\email{andrea.caranti@unitn.it} 

\urladdr{http://www.science.unitn.it/$\sim$caranti/}



\begin{abstract}
Let $G$ be a group, and $S(G)$ be the group of permutations on the set
$G$. The (abstract) holomorph of $G$ is the natural semidirect product
$\Aut(G) G$. We will write $\Hol(G)$ for the normalizer of the image in
$S(G)$ of the right regular representation of $G$,
\begin{equation*}
 \Hol(G) = N_{S  (G)}(\rho(G)) = \Aut(G) \rho(G) \cong \Aut(G) G,
\end{equation*}
and also refer to it as the holomorph of $G$. More generally,
if $N$ is any regular subgroup of $S(G)$, then
$N_{S(G)}(N)$ is isomorphic to the holomorph of $N$.

G.A.~Miller has shown that the group
\begin{equation*}
 T(G) = N_{S(G)}(\Hol(G))/\Hol(G)
\end{equation*}
acts regularly on the set of the regular subgroups $N$ of $S(G)$ which
are isomorphic  to $G$,  and have  the same holomorph  as $G$,  in the
sense that $N_{S(G)}(N) = \Hol(G)$.

If $G$ is non-abelian, inversion on $G$ yields an involution in
$T(G)$. Other non-abelian regular subgroups $N$ of $S(G)$ having the same
holomorph as $G$ yield (other) involutions
in $T(G)$.
In the cases studied in the literature, $T(G)$ turns out to be a
finite $2$-group, which is often elementary abelian. 

In this  paper we exhibit  an example of  a finite $p$-group  $\Gp$ of
class $2$, for  $p > 2$ a prime, which is  the smallest $p$-group such
that $T(\Gp)$ is non-abelian,  and not a $2$-group. Moreover, $T(\Gp)$
is not generated by involutions when $p > 3$.

More generally, we develop some aspects  of a theory of $T(G)$ for $G$
a finite $p$-group of class $2$, for $p > 2$. In particular, we show
that for such 
a group $G$ there is an element of order $p-1$ in $T(G)$, and exhibit
examples where $\Size{T(G)} = p - 1$, and others where $T(G)$
contains a large elementary abelian $p$-subgroup.
\end{abstract}

\thanks{To appear in J.\ Algebra \textbf{516} (2018),
352--372.
\endgraf
\texttt{10.1016/j.jalgebra.2018.09.031}
\endgraf
The author is a member of INdAM---GNSAGA. The author
gratefully acknowledges support from the Department of Mathematics of
the University of Trento and  from MIUR---Italy via PRIN 2015TW9LSR\_005.}

\maketitle

 \thispagestyle{empty}

 \bibliographystyle{amsalpha}

\section*{Introduction}

Let $G$ be a group, and $S(G)$ be the group of permutations on the set
$G$, under left-to-right composition. The image $\rho(G)$ of the right
regular representation of $G$ is 
a regular subgroup of $S(G)$, and its normalizer is the semidirect
product of $\rho(G)$ by the automorphism group $\Aut(G)$ of $G$,
\begin{equation*}
  N_{S(G)}(\rho(G)) 
  = 
  \Aut(G) \rho(G).
\end{equation*}
We will refer to this group, which is isomorphic to the \emph{abstract
  holomorph} $\Aut(G) G$ of $G$, as
the \emph{holomorph} $\Hol(G)$ of $G$.

More generally, if $N$ is a regular subgroup of $S(G)$, then
$N_{S(G)}(N)$ is isomorphic to the holomorph of $N$.
Let us thus consider the set
\begin{equation*}
  \Hc(G)
  =
  \Set{ N \le S(G) : \text{$N$ is regular, $N \cong G$ and $N_{S(G)}(N)
      = \Hol(G)$} }.
\end{equation*}
of the regular subgroups of $S(G)$ which are isomorphic to $G$, and
have in some sense the same holomorph as $G$.

G.A.~Miller has  shown~\cite{Miller-multi} that the so-called
\emph{multiple holomorph} of $G$ 
\begin{equation*}
  \NHol(G) 
  = 
  N_{S(G)}(\Hol(G))
  = 
  N_{S(G)}(N_{S(G)}(\rho(G)))
\end{equation*}
acts  transitively on $\Hc(G)$, and  thus the group
\begin{equation*}
  T(G)  = \NHol(G)  / \Hol(G)
\end{equation*}
acts regularly  on $\Hc(G)$.

Let $G$ be  a non-abelian group.  The inversion map  $\inv : g \mapsto
g^{-1}$ on $G$,  which conjugates $\rho(G)$ to  the image $\lambda(G)$
of the left regular representation $\lambda$, induces an involution in
$T(G)$.  More generally, let $N$ be a non-abelian, regular subgroup of
$G$  that  has  the  same  holomorph   as  $G$  (that  is,  such  that
$N_{S(G)}(N) =  \Hol(G) = N_{S(G)}(\rho(G))$), and  consider the group
$S(N)$.   As above,  the inversion  map  $\inv_{N}$ on  $N$ yields  an
involution in $T(N)$. The  bijection $N \to G$ that maps  $n \in N$ to
$1^{n}  \in  G$  yields  an  isomorphism $S(N)  \to  S(G)$  that  maps
$\rho(N)$ to  $N$.  Under  this isomorphism,  $\inv_{N}$ maps  onto an
involution  in $T(G)$.   So  each  such $N$  yields  an involution  in
$T(G)$, although  different $N$ need not  induce different involutions
in $T(G)$.

Recently T.~Kohl  has described~\cite{Kohl-multi} the  set $\Hc(G)$
and the group $T(G)$
for $G$ dihedral or generalized quaternion.
In~\cite{fgag}, we have redone,  via a commutative
ring connection, the work of
Mills~\cite{Mills-multi}, which determined $\Hc(G)$
and $T(G)$ for $G$ a finitely generated abelian
group. In~\cite{perfect} we have studied the case of finite
perfect groups. 

In  all these  cases, $T(G)$  turns out  to be  an elementary  abelian
$2$-group.  T.~Kohl  mentions in~\cite{Kohl-multi} two  examples where
$T(G)$  is   a  non-abelian  $2$-group;   it  can  be   verified  with
\textsf{GAP}~\cite{GAP4} that  for the finite $2$-groups  $G$ of order
up to $8$ the $T(G)$ are elementary abelian $2$-groups, and that there
are  exactly  two  groups  $G$  of order  $16$  such  that  $T(G)$  is
non-abelian: for  the first one  $T(G)$ is isomorphic to  the dihedral
group $D$ of order $8$, while  for the second one $T(G)$ is isomorphic
to the direct product  of $D$ by a group of order  $2$.  In a personal
communication, Kohl  has asked  whether $T(G)$  is always  a $2$-group
when $G$ is finite.

In this paper  we study the groups $T(G)$, for  $G$ a finite $p$-group
of   nilpotence  class   $2$,  where   $p$  is   an  odd   prime.   

In
Section~\ref{sec:powers} we show that for such a $G$, the group $T(G)$
contains      an      element     of      order      $p-1$
(Proposition~\ref{prop:powers}). In Section~\ref{sec:examples} we show
that this minimum  order is attained by the two  non-abelian groups of
order $p^{3}$, and by the finite $p$-groups that are free in the
variety of groups of class $2$ and exponent
$p$ (Proposition~\ref{thm:p3} and Theorem~\ref{thm:free}). 

In Section~\ref{sec:Tgp} we show that the
group
\begin{equation*}
  \Gp
  =
  \Span{
    x, y
    :
    x^{p^{2}}, y^{p^{2}},
    [x, y] = x^{p}
  },
\end{equation*}
of  order $p^{4}$  and class  $2$, has  $T(\Gp)$ of  order $p  (p-1)$,
isomorphic to  $\AGL(1, p)$, that is,  to the holomorph of  a group of
order $p$. Thus  $T(\Gp)$ is non-abelian, it is 
not  a  $2$-group, and  for  $p  > 3$  it  is  not even  generated  by
involutions, as  the subgroup generated  by the involutions  has index
$(p-1)/2$ in $T(\Gp)$.

Note that when $G$ is an  abelian $p$-group, with $p$  odd, $T(G) =
\Set{1}$ \cite[Theorem~3.1 and Lemma~3.2]{fgag}.  Therefore for $p > 3$
the non-abelian groups  of order $p^{3}$ are the  smallest examples of
finite $p$-groups  such that $T(G)$ is  not a $2$-group, and  $\Gp$ is
the smallest example of a finite $p$-group  $G$, for $p > 2$, such that
$T(G)$ is non-abelian, and not a $2$-group.

In Section~\ref{sec:examples}, besides  the examples already mentioned
above,  we exhibit  a  class of  finite $p$-groups  $G$  of class  $2$
(Theorem~\ref{thm:big})  for  which  $T(G)$  is  non-abelian,  and
contains a large elementary abelian $p$-subgroup.

Section~\ref{sec:prelim}   collects   some  preliminary   facts.    In
Section~\ref{sec:tools}    we     introduce    a     linear    setting
(Proposition~\ref{prop:Delta}) that simplifies the calculations in the
later sections, and  develop some more general aspects of  a theory of
$T(G)$ for finite $p$-groups $G$ of class $2$, for $p > 2$.

Regular subgroups of $S(G)$ correspond to right skew
braces structures on $G$ (see~\cite{skew}). Also, this work is related
to  the  enumeration  of  Hopf-Galois structures  on  separable  field
extensions,  as C.~Greither  and B.~Pareigis  have shown~\cite{GrePar}
that these  structures can be described through  those regular subgroups
of a suitable symmetric group, that are normalised by a given regular
subgroup;   this   connection   is    exploited   in   the   work   of
L.~Childs~\cite{Chi},    N.P.~Byott~\cite{Byo},    and    Byott    and
Childs~\cite{ByoChi}.

The system for computational discrete algebra \textsf{GAP}~\cite{GAP4}
has been invaluable for gaining the 
computational evidence which led to the results of this paper.

We are very grateful to the referee for several useful suggestions. We
are very  grateful to Giovanni  Deligios for pointing out  some rather
bad misprints (which  did not affect, however, the  correctness of the
paper); these have been corrected in version 5.07.

\section{Preliminaries}
\label{sec:prelim}

We recall some standard material from~\cite{perfect}, and complement
it with a couple
of Lemmas that will be useful in the rest of the paper.

Let $G$  be a group. Denote by $S(G)$ the group of permutations of the
set $G$, under left-to-right composition.

Let
\begin{equation*}
  \begin{aligned}
    \rho :\ &G \to S(G)
    \\&g \mapsto (h \mapsto h g) 
  \end{aligned}
\end{equation*}
be the right regular representation of $G$

\begin{definition}
  The \emph{holomorph} of $G$ is
  \begin{equation*}
    \Hol(G) = N_{S(G)}(\rho(G)) = \Aut(G) \rho(G).
  \end{equation*}
\end{definition}

\begin{theorem}\label{thm:normal-regular}
Let $G$ be a finite group. The following data are equivalent.
  \begin{enumerate}
  \item\label{item:N}
    A regular subgroup $N \norm \Hol(G)$.
  \item\label{item:gamma}
    A map $\gamma : G \to \Aut(G)$ such that for $g, h \in G$ and
    $\beta \in \Aut(G)$
    \begin{equation}\label{eq:gamma-for-normal}
      \begin{cases}
        \gamma(g h) = \gamma(h) \gamma(g),\\
        \gamma(g^{\beta}) = \gamma(g)^{\beta}.
      \end{cases}
    \end{equation}
  \item\label{item:braces} 
    A group operation $\circ$ on $G$ such that
    for $g, h, k \in G$ 
    \begin{equation*}
      \begin{cases}
      (g h) \circ k = (g \circ k) k^{-1} (h \circ k),\\
        \Aut(G) \le \Aut(G, \circ).
      \end{cases}
    \end{equation*}
  \end{enumerate}
  The data of~\eqref{item:N}-\eqref{item:braces} are related as follows. 
  \begin{enumerate}[(i)]\label{item:circ}
  \item
    $
      g \circ h = g^{\gamma(h)} h
    $
    for $g, h \in G$. 
   \item 
    Each element of $N$ can be written uniquely in the form $\nu(h) =
    \gamma(h) \rho(h)$, for some $h \in G$.
 \item 
    For $g, h \in G$ one has
    $
      g^{\nu(h)} = g \circ h.
    $
  \item\label{item:nu-is-iso}
    The map 
    \begin{align*}
    \setlength\arraycolsep{1.5pt}
    \begin{matrix}
          \nu: & (G, \circ) & \to & N\\
          & h &\mapsto &\gamma(h) \rho(h)
    \end{matrix}
    \end{align*}
    is an isomorphism.    
  \end{enumerate}
\end{theorem}

This is  basically~\cite[Theorem~5.2]{perfect}; we recall  briefly the
main points. 

Recall that a  subgroup $N \le S(G)$ is \emph{regular}  if for each $h
\in  G$ there  is a  unique  $\nu(h) \in  N$ such  that $1^{\nu(h)}  =
h$. (So  $\nu : G  \to N$ is  the inverse of  the bijection $N  \to G$
given by $n  \mapsto 1^{n}$, for $n \in N$.)  Given a regular subgroup
$N \norm \Hol(G)  = \Aut(G) \rho(G)$, there  is a map $\gamma  : G \to
\Aut(G)$ such that  $\nu(h) = \gamma(h) \rho(h)$,  as $1^{\gamma(h)} =
1$.      The     equivalence    of     item~\eqref{item:braces}     of
Theorem~\ref{thm:normal-regular} with  the previous ones  follows from
the theory of skew braces~\cite{Bach-braces}.

\begin{remark}
  In   the   following,   when    discussing   a   subgroup   $N$   as
  in~\eqref{item:N}   of  Theorem~\ref{thm:normal-regular},   we  will
  use the other notation of the Theorem
  without further mention.
\end{remark}

We introduce the sets
\begin{equation*}
  \Hc(G)
  =
  \Set{ N \le S(G) : \text{$N$ is regular, $N \cong G$ and $N_{S(G)}(N)
    = \Hol(G)$} }
\end{equation*}
and
\begin{equation*}
  \Jc(G)
  =
  \Set{ N \le S(G) : \text{$N$ is regular, $N \norm \Hol(G)$} }
  \supseteq 
  \Hc(G)
\end{equation*}
As in~\cite{fgag, perfect}, for the groups $G$ we consider we first
determine $\Jc(G)$, using Theorem~\ref{thm:normal-regular}, then check
which elements of $\Jc(G)$ lie in 
$\Hc(G)$, and finally compute $T(G)$, or part of it.

G.A.~Miller has  shown~\cite{Miller-multi} that $\Hc(G)$ is  the orbit
of $\rho(G)$ under conjugation by the \emph{multiple holomorph}
\begin{equation*}
  \NHol(G) 
  = 
  N_{S(G)}(\Hol(G))
  = 
  N_{S(G)}(N_{S(G)}(\rho(G)))
\end{equation*}
of $G$, so that  the group  
\begin{equation*}
  T(G) 
  = 
  \NHol(G) / \Hol(G)
  =
  N_{S(G)}(N_{S(G)}(\rho(G))) / N_{S(G)}(\rho(G))
\end{equation*}
acts regularly on $\Hc(G)$.

An element of $S(G)$
conjugating $\rho(G)$ to $N$ can be modified by a translation, and
assumed to fix $1$.   We have, using the notation $a^{b} = b^{-1} a b$
for conjugacy in a group, the following
\begin{lemma}\label{lemma:conjugation}
  Suppose $N \in \Hc(G)$, and let $\theta \in \NHol(G)$ such that
  $\rho(G)^{\theta} = N$ and $1^{\theta} = 1$. Then
  \begin{equation*}
    \theta : G \to (G, \circ)
  \end{equation*}
  is an isomorphism such that
  \begin{equation}\label{eq:rho-theta-nu}
    \rho(g)^{\theta} = \nu(g^{\theta}) = \gamma(g^{\theta}) \rho(g^{\theta}),
  \end{equation}
  for $g \in G$.

  Conversely, an isomorphism $\theta : G \to (G, \circ)$ conjugates
  $\rho(G)$ to $N$.
\end{lemma}

This is~\cite[Lemma~4.2]{perfect}.

The next Lemma shows that an isomorphism $\theta$ as in
Lemma~\ref{lemma:conjugation} determines $\gamma$ uniquely.
\begin{lemma}\label{lemma:from-theta-to-gamma}
  Let $N \in \Hc(G)$.
  If
  \begin{equation*}
    \theta : G \to (G, \circ),
  \end{equation*}
  is an isomorphism, then 
  \begin{equation}\label{eq:theta-gamma}
    (g h)^{\theta}
    =
    (g^{\theta})^{\gamma(h^{\theta})} h^{\theta},
  \end{equation}
  for $g, h \in G$. It follows that the  $\gamma$ associated to $N$ is given by
  \begin{equation*}
    g^{\gamma(h)}
    =
    (g^{\theta^{-1}} h^{\theta^{-1}})^{\theta} h^{-1}.
  \end{equation*}
\end{lemma}

\begin{proof}
  The first formula comes from the definition of the operation
  $\circ$ in Theorem~\ref{thm:normal-regular}\eqref{item:circ}, and
  the second one follows immediately from it. 
\end{proof}

In  the next  lemma  we note  that  $\gamma$ is  well  defined for  the
elements of $T(G)$.
\begin{lemma}\label{lemma:theta-by-alpha}
  Let $N \in \Hc(G)$, and let  $\theta: G \to (G, \circ)$ be an
  isomorphism. 

  If $\alpha  \in \Aut(G)$, Then  $\alpha \theta$ represents  the same
  element of $T(G)$, and has the same associated $\gamma$.
\end{lemma}

\begin{proof}
  It suffices to prove the second statement, which follows from the
  fact that the $\gamma$ associated to $\alpha \theta$ is given,
  according to Lemma~\ref{lemma:from-theta-to-gamma}, by
  \begin{align*}
    (g^{(\alpha \theta)^{-1}} h^{(\alpha \theta)^{-1}})^{\alpha \theta} h^{-1}
    &=
    (g^{\theta^{-1} \alpha^{-1}} h^{\theta^{-1}
      \alpha^{-1}})^{^{\alpha \theta}} h^{-1}
    \\&=
    (g^{\theta^{-1}} h^{\theta^{-1}})^{\theta}
    h^{-1}
    \\&=
    g^{\gamma(h)}.
  \end{align*}
\end{proof}


For a group $G$, denote by $\iota$ the morphism
\begin{align*}
  \iota :\ &G \to \Aut(G)\\
           &g \mapsto (h \mapsto h^{g} = g^{-1} h g)
\end{align*}
that maps $g \in G$ to the inner automorphism $\iota(g)$ it
induces. The following Lemma is proved in~\cite[Section~6]{perfect}.
\begin{lemma}\label{lemma:formulas}
  Let $G$ be a group, and let $N \norm \Hol(G)$ be a regular subgroup
  of $S(G)$. 

  The following general formulas hold in $\Aut(G) G$, for $g, h \in G$
  and $\beta \in \Aut(G)$.
  \begin{enumerate}
  \item\label{item:self-adjoint-on-autos}
    $\gamma([\beta, g^{-1}]) = [\gamma(g), \beta]$,
  \item\label{item:gamma-of-commutators}
      $\gamma([h, g^{-1}]) = \iota([\gamma(g), h])$,
  \end{enumerate}
\end{lemma}

Note that~\eqref{item:gamma-of-commutators} is an instance
of~\eqref{item:self-adjoint-on-autos}, for $\beta = \iota(h)$.

\section{Some basic tools}
\label{sec:tools}

If $G$ is a group of nilpotence class two, we will use repeatedly the
standard identity
\begin{equation*}
  (g h)^{n} = g^{n} h^{n} [h, g]^{\binom{n}{2}},
\end{equation*}
valid for $g , h \in G$ and $n \in \N$ \cite[5.3.5]{Rob1996}.

We write $\Aut_{c}(G)$ for the group of central automorphisms of $G$,
that is
\begin{equation*}
  \Aut_{c}(G)
  =
  \Set{ \alpha \in \Aut(G) : [G, \alpha] \subseteq Z(G) },
\end{equation*}
where the commutator is taken in $\Aut(G) G$.

\begin{lemma}\label{lemma:eqcond}
  Let $G$ be a finite $p$-group of class two, and $N \norm \Hol(G)$ a
  regular subgroup.

  The following are equivalent.
  \begin{enumerate}
  \item\label{item:1}
    $G' \le \ker(\gamma)$, that is, $\gamma(G') = 1$,
  \item\label{item:2}
    $\gamma(G)$ is abelian,
  \item\label{item:3}
    $[\gamma(G), G] \le Z(G)$, that is, $\gamma(G) \le
    \Aut_{c}(G)$,
  \item\label{item:4}
    $[\gamma(G), G] \le \ker(\gamma)$.
  \end{enumerate}
  Moreover, these conditions imply $[G', \gamma(G)] = 1$.
\end{lemma}

\begin{proof}
  \eqref{item:1} and \eqref{item:2} are clearly equivalent.

  Lemma~\ref{lemma:formulas}\eqref{item:gamma-of-commutators} yields
  that $\gamma(G') = 1$ iff $[G, \gamma(G)] \le Z(G)$, that is,
  \eqref{item:1} is equivalent to~\eqref{item:3}.

  Setting $\beta = \gamma(h)$
  in~\ref{lemma:formulas}\eqref{item:self-adjoint-on-autos}, 
  for $h \in G$, we get 
  \begin{align*}
    \gamma( [\gamma(h), g^{-1}] )
    =
    [\gamma(g), \gamma(h)],
  \end{align*}
  which shows that \eqref{item:2} and~\eqref{item:4} are equivalent.

  As to the last statement, it is a well known and elementary fact that central
  automorphisms centralise the derived subgroup.
\end{proof}

We  now   introduce  a  linear   technique  that  will   simplify  the
calculations in the next sections. Here  and in the following, we will
occasionally employ additive notation  for the abelian groups $G/Z(G),
G/G'$ and $Z(G)$.
\begin{proposition}\label{prop:Delta}
  Let $G$ be a finite $p$-group of class two, for $p > 2$.

  There is a one-to-one correspondence between
  \begin{enumerate}
  \item\label{item:gamma-setting} the regular subgroups $N \norm \Hol(G)$ 
  such that
  \begin{enumerate}
  \item\label{item:a-assumption} $\gamma(G) \le \Aut_{c}(G)$, and
  \item\label{item:b-assumption} $[Z(G), \gamma(G)] = 1$,
  \end{enumerate}
  and
  \item\label{item:Delta}
    the bilinear maps
    \begin{align*}
      \Delta : G/Z(G) \times G/G' \to Z(G)
    \end{align*}
    such that   
    \begin{equation}\label{eq:Delta-beta}
    \Delta(g^{\beta}, h^{\beta}) = \Delta(g, h)^{\beta}
  \end{equation}
  for $g \in G/Z(G)$, $h \in G/G'$ and $\beta \in \Aut(G)$.
  \end{enumerate}
  The correspondence is given by
  \begin{equation}\label{eq:gamma-and-Delta}
    \Delta(g Z(G), h G') = [g, \gamma(h)],
  \end{equation}
  for $g, h \in G$.
\end{proposition}

\begin{remark}\label{rem:commutator-map}
  Clearly the commutator maps
  \begin{equation*}
    (g Z(G), h G') \mapsto [g, h]^{c},
  \end{equation*}
  for    some     fixed integer   $c$,    satisfy     the    conditions
  of~\eqref{item:Delta}. See Section~\ref{sec:powers} for a discussion
  of the corresponding regular subgroups.
\end{remark}

\begin{proof}[Proof of Proposition~\ref{prop:Delta}]
  Let  us start  with  the  setting of~\eqref{item:gamma-setting}.  By
  assumption~\eqref{item:a-assumption}, and Lemma~\ref{lemma:eqcond},
  we have $[G, \gamma(G)] \le Z(G)$. 

  If $z \in Z(G)$ we have
  \begin{equation*}
    [g z, \gamma(h)] 
    =
    [g, \gamma(h)]^{z} [z, \gamma(h)]
    =
    [g, \gamma(h)],
  \end{equation*}
  as by assumption~\eqref{item:b-assumption} $[Z(G), \gamma(G)] =
  1$. If $w \in G'$ we have 
  \begin{equation*}
    [g, \gamma(h w)] 
    =
    [g, \gamma(w) \gamma(h)]
    =
    [g, \gamma(h)],
  \end{equation*}
  as assumption~\eqref{item:a-assumption} implies $G' \le
  \ker(\gamma)$ by Lemma~\ref{lemma:eqcond}. Therefore the map 
  \begin{align*}
    \Delta : G/Z(G) \times G/G' \to Z(G)
  \end{align*}
  of~\eqref{eq:gamma-and-Delta},  induced  by   $(g,  h)  \mapsto  [g,
    \gamma(h)]$, is well defined.

  We now prove  that $\Delta$ is bilinear. 
  For $g, h \in G/Z(G)$, and $k \in G/G'$ we have
  \begin{align*}
    \Delta(g + h, k) 
    &=
          [g + h, \gamma(k)]
          \\&=
            [g, \gamma(k)]^{h} + [h, \gamma(k)]
            \\&=
              [g, \gamma(k)] + [h, \gamma(k)]
  \\&=
  \Delta(g, k) + \Delta(h, k),
  \end{align*}
  since $[G, \gamma(G)] \le Z(G)$. For $g \in G/Z(G)$ and $h, k \in
  G/G'$ we have
  \begin{align*}
    \Delta(g, h + k)
    &=
          [g, \gamma(h + k)]
          \\&=
            [g, \gamma(k) + \gamma(h)]
            \\&=
              [g, \gamma(h)] + [g, \gamma(k)]^{\gamma(h)}
              \\&=
                [g, \gamma(h)] + [g, \gamma(k)]
                \\&=
                \Delta(g, h) + \Delta(g, k),
  \end{align*}
  since $[G, \gamma(G)] \le Z(G)$, and $[Z(G), \gamma(G)] = 1$.
  
  To prove~\eqref{eq:Delta-beta} we compute, for $g \in G/Z(G)$, $h
  \in G/G'$ and $\beta \in \Aut(G)$,
  \begin{align*}
    \Delta(g^{\beta}, h^{\beta})
    &=
          [g^{\beta}, \gamma(h^{\beta})]
          \\&=
            [g^{\beta}, \gamma(h)^{\beta}]
            \\&=
            g^{-\beta} (g^{\beta})^{\beta^{-1} \gamma(h) \beta}
            \\&=
            g^{-\beta} g^{\gamma(h) \beta}
            \\&=
              [g, \gamma(h)]^{\beta}
            \\&=
            \Delta(g, h)^{\beta}.
  \end{align*}

  Conversely, let $\Delta$ be as in~\eqref{item:Delta}, and define
  a function $\gamma$ on $G$ via~\eqref{eq:gamma-and-Delta}, that is
  \begin{equation*}
    g^{\gamma(h)} = g \cdot \Delta(g Z(G), h G').
  \end{equation*}
  The fact that $\Delta$ is linear in the first component, and takes
  values in the centre, 
  implies that 
  $\gamma(h) \in \Aut(G)$ for every $h \in G$. The fact that $\Delta$
  is linear in the second component implies that $\gamma : G \to
  \Aut(G)$ is a morphism.  It now follows immediately
  from~\eqref{eq:gamma-and-Delta} that $[G,
    \gamma(G)] \le Z(G)$, so that by
  Lemma~\ref{lemma:eqcond} $\gamma$
  satisfies~\eqref{item:a-assumption} and the
  first  condition of~\eqref{eq:gamma-for-normal}. Also, $[Z(G),
    \gamma(G)] = 1$, so that  
  $\gamma$
  satisfies~\eqref{item:b-assumption}.
  Moreover, \eqref{eq:Delta-beta} implies that
  \begin{align*}
    g^{\gamma(h)^{\beta}}
    =
    g^{\beta^{-1} \gamma(h) \beta}
    =
    (g^{\beta^{-1}} \Delta(g^{\beta^{-1}} Z(G), h G'))^{\beta}
    =
    g \Delta(g Z(G), h^{\beta} G')
    =
    g^{\gamma(h^{\beta})},
  \end{align*}
  that is, the second condition of~\eqref{eq:gamma-for-normal} holds.
\end{proof}

We now record two consequences of Proposition~\ref{prop:Delta}
concerning commutators and powers.

\begin{lemma}\label{lemma:commutator-of-nu}
  Let $G$ be a finite $p$-group of class two, and $N \norm \Hol(G)$ a
  regular subgroup.

  Suppose
  \begin{enumerate}
  \item $\gamma(G) \le \Aut_{c}(G)$, and
  \item $[Z(G), \gamma(G)] = 1$.
  \end{enumerate}

  Then for $g, h \in G$ we have
  \begin{equation}\label{eq:the-id}
    \begin{aligned}\null
      [\nu(g), \nu(h)] 
      &=
      \nu( [g, h] [g, \gamma(h)] [h, \gamma(g)]^{-1} )
      \\&=
      \nu([g, h] + \Delta(g, h) - \Delta(h, g)).
    \end{aligned}
  \end{equation}
\end{lemma}

\begin{proof}
  Note that $\nu(g)^{-1} : k \mapsto k^{\gamma(g)^{-1}}
  g^{-\gamma(g)^{-1}}$.
  
  Since an element $\nu(x)$ of the subgroup $N$ is determined by the
  element $x$ to which it takes $1$, it suffices to compute the
  following
  \begin{align*}
    1^{[\nu(g), \nu(h)]}
    &=
    ((((( g^{-1})^{\gamma(g)^{-1}})^{\gamma(h)^{-1}}
    (h^{-1})^{\gamma(h)^{-1}})^{\gamma(g)} x)^{\gamma(h)}) y
    \\&=
    (g^{-1})^{[\gamma(g), \gamma(h)]}
    (h^{-1})^{\gamma(h)^{-1} \gamma(g) \gamma(h)}
    g^{\gamma(h)}
    h
    \\&=
    g^{-1}
    (h^{-1})^{\gamma(g)}
    g^{\gamma(h)}
    h
    \\&=
    [g, h] [h^{-1}, \gamma(g)] [g, \gamma(h)]
    \\&=
    [g, h] [g, \gamma(h)] [h, \gamma(g)]^{-1} ,
  \end{align*}
  where  we have  used  the facts  that $[G,  \gamma(G)]  \le Z(G)$  and
  $\gamma(G)$ is abelian, according to Lemma~\ref{lemma:eqcond}.
\end{proof}

\begin{lemma}\label{lemma:powers}
  Let $G$ be a finite $p$-group of class two, and $N \norm \Hol(G)$ a
  regular subgroup.
  
  Suppose
  \begin{enumerate}
  \item $\gamma(G) \le \Aut_{c}(G)$, and
  \item $[Z(G), \gamma(G)] = 1$.
  \end{enumerate}

  Then for $g \in G$ we have
  \begin{align*}
    \nu(g)^{n} 
    = 
    \nu( (g^{n})^{\gamma(g^{(n-1)/2})})
    =
    \nu( g^{n} \cdot \Delta(g, g)^{\binom{n}{2}} )
    .
  \end{align*}
  In particular, the elements of $G$ retain their orders under $\nu$.
\end{lemma}

\begin{proof}
  This follows from
  \begin{align*}
    1^{\nu(g)^{n}}
    &=
    g^{\gamma(g)^{n-1}} g^{\gamma(g)^{n-2}} \cdots g^{\gamma(g)} g
    \\&=
    g^{n} [g, \gamma(g)^{n-1}] [g, \gamma(g)^{n-2}] \cdots [g, \gamma(g)]
    \\&=
    g^{n} [g, \gamma(g)]^{\binom{n}{2}}
    =
    g^{n} \cdot \Delta(g, g)^{\binom{n}{2}}    
    \\&=
    g^{n} [g^{n}, \gamma(g)^{(n-1)/2}]
    \\&=
    (g^{n})^{\gamma(g)^{(n-1)/2}},
  \end{align*}
  where we have used the fact that $[G,  \gamma(G)]  \le Z(G)$,
  according to Lemma~\ref{lemma:eqcond}. 

  As to  the last  statement, for $n$ odd $g^{n}  = 1$  implies $\Delta(g,
  g)^{\binom{n}{2}}  =  \Delta(g^{n},  g)^{(n-1)/2}   =  1$,  so  that
  $\nu(g)^{n} = 1$. Conversely, if $\nu(g)^{n}  = 1$, we have $g^{n} =
  \Delta(g,  g)^{-   \binom{n}{2}}  \in  Z(G)$,  so   that  $\Delta(g,
  g)^{\binom{n}{2}} = \Delta(g^{n}, g)^{(n-1)/2} = 1$, and thus $g^{n}
  = 1$.
\end{proof}

\section{Power Isomorphisms}
\label{sec:powers}

Let $G$ be a finite $p$-group of class two, for $p > 2$.

As noted in Remark~\ref{rem:commutator-map}, the maps
\begin{align*}
  \Delta_{c} :\ &G/Z(G) \times G/G' \to G'\\
                     &(g Z(G), h G') \mapsto [g, h]^{c},
\end{align*}
for integers $c$, satisfy the conditions~\eqref{item:Delta} of
Proposition~\ref{prop:Delta}. To $\Delta_{c}$ we then associate
the $\gamma_{c}$ given by
\begin{equation*}
  g^{\gamma_{c}(h)} 
  = 
  g \Delta_{c}(g, h)
  =
  g [g, h]^{c}
  =
  g^{h^{c}},
\end{equation*}
that is, 
\begin{equation}\label{eq:gamma-from-Delta}
  \gamma_{c}(h) = \iota(h^{c}).  
\end{equation}

All these $\gamma_{c}$ yields regular subgroups $N \norm \Hol(G)$.
To see which $\gamma_{c}$ yield subgroups $N \in \Hc(G)$,
consider, for integers $d$  coprime to $p$, the bijection 
$\theta_{d} \in S(G)$  given by the $d$-th power map,  $\theta_{d} : x
\mapsto x^{d}$  on $G$, and write  $d'$ for the inverse  of $d$ modulo
$\exp(G)$.

For $g, h \in G$ we have
\begin{equation*}
  h^{\rho(g)^{\theta_{d}}}
  =
  (h^{d'})^{\rho(g) \theta_{d}}
  =
  (h^{d'} g)^{d}
  =
  h g^{d} [g, h^{d'}]^{\binom{d}{2}}
  =
  h^{\iota(g^{(1 - d)/2}) \rho(g^{d})},
\end{equation*}
so that
\begin{equation}\label{eq:rho-to-theta_d}
  \rho(g)^{\theta_{d}} 
  = 
  \iota(g^{(1 - d)/2}) \rho(g^{d}) \in \Aut(G) \rho(G) 
  =
  \Hol(G).  
\end{equation}
Since we have also
\begin{equation*}
  [\Aut(G), \theta_{d}] = 1,
\end{equation*}
it follows that $\theta_{d} \in \NHol(G) = N_{S(G)}(\Hol(G))$, and thus
$\theta_{d}$ induces an element of $T(G)$. To determine the
$\gamma$ associated to $\theta_{d}$, one can
apply~\eqref{eq:rho-theta-nu} to~\eqref{eq:rho-to-theta_d} to get
$\gamma(g^{\theta_{d}}) = \gamma(g^{d}) = \iota(g^{(1-d)/2})$, and
then, replacing $g$ by $g^{d'}$,
\begin{equation*}
  \gamma(g)
  =
  \iota(g^{(d'- 1)/2}).
\end{equation*}
(We could have also used
Lemma~\ref{lemma:from-theta-to-gamma} to obtain the same result.)

Clearly for a given $c$ the equation  $c = (d' - 1)/2$ has a
solution $d$  coprime to $\exp(G)$ if  and only if $c  \ne -1/2$
modulo $\exp(G)$,  so that for these  $c$ the $\gamma_{c}$
of~\eqref{eq:gamma-from-Delta} correspond to elements $N \in
\Hc(G)$. (The regular subgroup corresponding to $\gamma_{-1/2}$ is
abelian, and thus not isomorphic to $G$, see Remark~\ref{rem:BCH}.)

\eqref{eq:rho-to-theta_d} shows that the $\theta_{d}$ are all distinct
as $d$ ranges in the 
integers coprime to
$p$ between $1$ and $\exp(G/Z(G)) - 1$. Clearly $\theta_{d_{1}}
\theta_{d_{2}} = \theta_{d_{1} d_{2}}$, 
so that the $\theta_{d}$ yield a subgroup of $T(G)$ isomorphic to the
units of the integers modulo $\exp(G/Z(G))$.
We have obtained
\begin{proposition}\label{prop:powers}
  Let $G$ be a finite $p$-group of class two, for $p > 2$.

  Let $p^{r} = \exp(G/Z(G))$.

  Then $T(G)$ contains a cyclic subgroup of order $\phi(p^{r}) = (p-1)
  p^{r-1}$. In particular, $T(G)$ contains an element of order $p-1$.
\end{proposition}

\section{Computing $T(\Gp)$}
\label{sec:Tgp}

Let $p > 2$ be a prime, and define
\begin{equation*}
  \Gp
  =
  \Span{
    x, y
    :
    x^{p^{2}}, y^{p^{2}},
    [x, y] = x^{p}
  }.
\end{equation*}
This is a group of order $p^{4}$ and nilpotence class two, such that
$\Gp' = \Span{x^{p}}$ has order 
$p$, and $\Gp^{p} = \Frat(\Gp) = Z(\Gp) = \Span{ x^{p}, y^{p} }$ has
order $p^{2}$. 

It is a well-known fact and  an easy exercise that $\Aut(\Gp)$
induces  on the  $\F_{p}$-vector space  $V  = \Gp/\Frat(\Gp)$  the group  of
matrices
\begin{equation*}
  \Set{
  \begin{bmatrix}
    a & 0\\
    b & 1
  \end{bmatrix}
  :
  a \in \F_{p}^{*}, b \in \F_{p}
  }
\end{equation*}
with respect to the basis induced by $x, y$. (Our automorphisms
operate on the right, therefore our vectors are row
vectors.) The same group is induced
on the vector  space $W = \Gp^{p} = Z(\Gp) = \Frat(\Gp)$, with respect
to the basis given 
by $x^{p}, y^{p}$.

We first note that
Lemma~\ref{lemma:formulas}\eqref{item:self-adjoint-on-autos} implies 
\begin{lemma}
  Let $G$ be a group, and $N \norm \Hol(G)$ a regular  subgroup.

  Let $\alpha \in Z(\Aut(G))$.

  Then
  \begin{equation*}
    [G, \alpha] = \Set{ g^{-1} g^{\alpha} : g \in G } \le \ker(\gamma).
  \end{equation*}
\end{lemma}

In the group $\Gp$, the power map $\alpha : g \mapsto g^{1+p}$ is an
automorphism lying in the centre of $\Aut(\Gp)$. Hence
\begin{lemma}\label{lemma:kernel-contains}
  Let $N \in \Jc(\Gp)$.
  Then
  \begin{equation*}
    \Gp' \le Z(\Gp) = \Frat(\Gp) = \Gp^{p} \le \ker(\gamma).
  \end{equation*}
\end{lemma}
It follows that $\Gp$ satisfies the equivalent conditions of
Lemma~\ref{lemma:eqcond}. In particular, $\gamma(\Gp) \le
\Aut_{c}(\Gp)$. Moreover, 
\begin{equation}\label{eq:gamma-centralizes}
  [Z(\Gp), \gamma(\Gp)] 
  = 
  [\Gp^{p}, \gamma(\Gp)] 
  = 
  [\Gp, \gamma(\Gp)]^{p} 
  = 
  1,
\end{equation}
as $[\Gp, \gamma(\Gp)]
\le Z(\Gp)$, and $Z(\Gp)$ has exponent $p$.

We may thus appeal to the $\Delta$ setting of
Proposition~\ref{prop:Delta}. Since $\gamma(\Frat(\Gp)) = 1$ by
Lemma~\ref{lemma:kernel-contains}, 
$\Delta$ is well-defined as a map 
$V \times V \to W$.

For $\Delta$ as in Proposition~\ref{prop:Delta}, and $\beta \in
\Aut(G)$ which induces   
\begin{equation}\label{eq:beta}
  \begin{bmatrix}
    a & 0\\
    b & 1
  \end{bmatrix},
\end{equation}
on $V$, we have
\begin{align*}
  \Delta(x, x)^{\beta}
  =
  \Delta(x^{\beta}, x^{\beta})
  =
  \Delta(a x, a x)
  =
  a^{2} \Delta(x, x).
\end{align*}
If $p >  3$, and we choose  $a \ne 1,  -1$, we obtain  that if
$\Delta(x, x) \ne 0$, then $\Delta(x, x) \in W$ is an eigenvector for
$\beta$ with  respect to the eigenvalue  $a^{2} \ne a, 1$.  It follows
that $\Delta(x, x) = 0$.

If $p = 3$, we first choose $a = -1$ and $b = 0$ to obtain that
$\Delta(x, x) \in \Span{y^{p}}$. We then choose $a = -1$ and $b = 1$
to obtain that $\Delta(x, x) \in \Span{ x^{p} y^{2 p}}$, so that
$\Delta(x, x) = 0$ in this case too. (We could have used this
argument also for $p > 3$.)

We have also
\begin{align*}
  \Delta(x, y)^{\beta}
  =
  \Delta(x^{\beta}, y^{\beta})
  =
  \Delta(a x, b x + y)
  =
  a \Delta(x, y).
\end{align*}
Taking $b = 0$ and $a \ne 1$ in $\beta$, we see that $\Delta(x, y) = t
x^{p}$ for some $t$.

Similarly, $\Delta(y, x) = s x^{p}$ for some $s$.
We have also
\begin{align*}
  \Delta(y, y)^{\beta}
  =
  \Delta(y^{\beta}, y^{\beta})
  =
  \Delta(b x + y, b x + y)
  =
  b (s + t) x^{p} + \Delta(y, y).
\end{align*}
If $\Delta(y, y) = u y^{p} + v x^{p}$, we have
\begin{align*}
  u y^{p} + (u b  + v a) x^{p}
  =
  \Delta(y, y)^{\beta}
  =
  (b (s + t) + v) x^{p} + u y^{p}.
\end{align*}
Setting $a = 1$ we get $u = s + t$, and then setting $a = -1$ we get
$v = 0$.

Therefore the $\gamma$ for $\Gp$ are given by
\begin{equation}\label{eq:gamma}
  \gamma_{s, t}(x)
  :
  \begin{cases}
    x \mapsto x\\
    y \mapsto x^{p s} y
  \end{cases}
  \qquad
  \gamma_{s, t}(y)
  :
  \begin{cases}
    x \mapsto x^{1 + p t}\\
    y \mapsto y^{1 + p (s + t)}
  \end{cases}
\end{equation}
for $s, t \in \F_{p}$.

Lemma~\ref{lemma:commutator-of-nu} yields that for the regular
subgroup $N \norm
\Hol(G)$ corresponding to $\gamma_{s, t}$ one has
\begin{align*}
  [\nu(x), \nu(y^{d})]
  &=
  \nu( [x, y^{d}] [x, \gamma_{s, t}(y^{d})] [y^{d}, \gamma_{s, t}(x)]^{-1} )
  \\&=
  \nu( x^{p d} x^{p d t} x^{-p d s} )
  \\&=
  \nu([x, y]^{d (1 + t - s)})
  \\&=
  \nu(x^{p d (1 + t - s)}).
\end{align*}
Therefore for $t - s + 1 = 0$ we have $\Gp \not\cong
(\Gp, \circ)$, as the 
latter is abelian, and the corresponding regular subgroup $N$ lies in
$\Jc(G) \setminus \Hc(G)$.

If $t - s + 1 \ne 0$, we choose 
$d = (1 + t - s)'$, so that 
$[\nu(x), \nu(y^{d})] = \nu(x^{p})$. Since by Lemma~\ref{lemma:powers}
$\nu(x^{p}) = \nu(x)^{p}$, and
$\Span{\nu(x)}, \Span{\nu(y^{d})}$ have each order $p^{2}$, and
intersect trivially, we have that $\Gp \cong (\Gp, \circ)$, via the
isomorphism defined by
\begin{equation}\label{eq:theta}
  \theta_{d, s} :
  \begin{cases}
    x \mapsto x\\
    y \mapsto y^{d}
  \end{cases}
\end{equation}
and by the accompanying $\gamma_{s, t}$ as above, for $t = d' + s - 1$.

\begin{remark}
  One might wonder where the power isomorphisms $\theta_{d}$ of
  Section~\ref{sec:powers} have gone in~\eqref{eq:theta}. They can be
  recovered using Lemma~\ref{lemma:theta-by-alpha}, with $\alpha =
  \beta$ defined on $V$ by~\eqref{eq:beta}, for $a = d$ and $b = 0$,
  and $\theta = \theta_{d, (1 - d')/2}$. 
\end{remark}

We now determine the structure of the group $T(\Gp)$. Note first
that we may take $d \in \F_{p}^{*}$ in~\eqref{eq:theta}.
To see this, apply Lemma~\ref{lemma:theta-by-alpha} with
$\alpha$ chosen to be the central automorphism defined by
\begin{equation*}
  \alpha :
  \begin{cases}
    x \mapsto x\\
    y \mapsto y^{1 + p d'u}
  \end{cases}
\end{equation*}
for some $u$. We obtain that
\begin{equation*}
  \alpha \theta_{d, s} :
  \begin{cases}
    x \mapsto x\\
    y \mapsto y^{d + p u}
  \end{cases}
\end{equation*}
yields the same element of $T(\Gp)$ with respect to the same $\gamma$.

We now want to show
\begin{theorem}
  \begin{equation*}
    \theta_{d, s} \theta_{e, u} \equiv \theta_{d e, s e' + u}.
  \end{equation*}
  Therefore  
  \begin{equation*}
    T(\Gp) = \Set{ \theta_{d, s} : d \in \F_{p}^{*}, s \in \F_{p} },
  \end{equation*}
  a group  of  order  $p  (p-1)$, is  isomorphic  to
  $\AGL(1, p)$, that is, to the holomorph of a group of order $p$.
\end{theorem}

\begin{proof}
  Note first  that it follows from Lemma~\ref{lemma:powers}, and
  $x^{\theta_{d, s}}  = x$ 
  and $x^{\gamma_{s, t}(x)} = x$  that
  \begin{equation}\label{eq:theta-on-powers-of-g}
    (x^{i})^{\theta_{d, s}} = x^{i}
  \end{equation}
  for all $i$.
  
  Similarly, note first that combining
  Theorem~\ref{thm:normal-regular}\eqref{item:nu-is-iso} 
  and Lemma~\ref{lemma:conjugation}
  we obtain that for all $i$
  \begin{align*}
    \nu((y^{i})^{\theta_{d,s}})
    =
    \nu(y)^{i},
  \end{align*}
  so that Lemma~\ref{lemma:powers} yields
  \begin{equation}\label{eq:theta-on-powers-of-h}
    \begin{aligned}
      (y^{i})^{\theta_{d, s}}
      &=
      (y^{d i})^{\gamma(y^{d (i-1)/2})}
      \\&=
      y^{(d i) (1 + p d (i-1)/2 (s + t)}
      \\&=
      y^{d i + p d \binom{i}{2} (s + t)}.
    \end{aligned}
  \end{equation}
  
  Now consider another $\theta_{e, u}$, and let $v = e' + u - 1$. 
  From~\eqref{eq:theta-on-powers-of-h} we obtain
  \begin{align*}
    y^{\theta_{d, s} \theta_{e, u}}
    =
    (y^{d})^{\theta_{e, u}}
    =
    y^{d e + p \binom{d}{2} e (u + v)}.
  \end{align*}
  Using Lemma~\ref{lemma:theta-by-alpha}  again, 
  we obtain
  $\alpha \theta_{d, s}  \theta_{e, u} = \theta_{d e, w}$  for some $w$,
  so that we can take $f = d e$. Set $z = f' + w - 1$.
  
  To determine $w$, we compute $(h g)^{\theta_{d, s} \theta_{e, u}}$ in
  two ways. We compute first
  \begin{equation}\label{eq:one}
    \begin{aligned}
      (y x)^{\theta_{d, s} \theta_{e, u}}
      &=
      (y^{\theta_{d, s} \theta_{e, u}})^{\gamma_{w, z}(x)} x^{\theta_{d, s}
        \theta_{e, u}}
      \\&=
      ((hy^{d})^{\theta_{e, u}})^{\gamma_{w, z}(x)} x
      \\&=
      (y^{d e})^{\gamma_{w, z}(x)} x y^{p \binom{d}{2} e (u + v)},
    \end{aligned}
  \end{equation}
  where we have used~\eqref{eq:theta-on-powers-of-g},
  \eqref{eq:theta-on-powers-of-h}, and~\eqref{eq:gamma-centralizes}.
  
  We then compute
  \begin{equation}\label{eq:two}
    \begin{aligned}
      (y x)^{\theta_{d, s} \theta_{e, u}}
      &= 
      ((y^{\theta_{d, s}})^{\gamma_{s, t}(x)} x)^{\theta_{e, u}}
      \\&=
      (y^{d} x^{p s d} x)^{\theta_{e, u}}
      \\&=
      ((y^{d})^{\theta_{e, u}})^{\gamma_{u, v}(x)} (x^{1 + p s
        d})^{\theta_{e, u}}
      \\&=
      (y^{d e + p \binom{d}{2} e (u + v)})^{\gamma_{u, v}(x)} (x^{1 + p s
        d})^{\theta_{e, u}}
      \\&=
      y^{d e + p \binom{d}{2} e (u + v)} x^{1 + p s d + p d e u},
    \end{aligned}
  \end{equation}
  where we have used~\eqref{eq:gamma-centralizes}.
  
  From~\eqref{eq:one}~and \eqref{eq:two} we obtain
  \begin{align*}
    (y^{d e})^{\gamma_{w, z}(x)}
    &=
    y^{d e} x^{p (d s + d e u)}.
  \end{align*}
  Taking ($d' e'$)-th powers we get
  \begin{equation*}
    y^{\gamma_{w, z}(x)}
    =
    y x^{p (s e' + u)}
  \end{equation*}
  so that
  \begin{equation*}
    \theta_{d, s} \theta_{e, u} \equiv \theta_{d e, s e' + u}
  \end{equation*}
  as claimed.
\end{proof}

\section{More Examples}
\label{sec:examples}

In this section we first show that  for $p > 2$ the non-abelian groups
$G$ of order $p^{3}$, and the groups $G$ of class $2$ and exponent $p$
have    $T(G)$   of    order    $p-1$,   that    is,   according    to
Proposition~\ref{prop:powers}, as small as possible.

We then exhibit a family of finite $p$-groups $G$ of class $2$, for $p
> 2$  with  minimum  number  of  generators $n \ge 4$  for  which  $T(G)$  is
non-abelian,  and contains  an  elementary abelian  subgroup of  order
$p^{\binom{n}{2} \binom{n+1}{2}}$.

\subsection{A group of order $p^{3}$}
\label{subsec:p3}

Let $p > 2$ be a prime, and $\Hp$ the group of order $p^{3}$ and
exponent $p^{2}$, so that
\begin{equation*}
  \Hp
  =
  \Span{
    x, y
    :
    x^{p^{2}}, y^{p},
    [x, y] = x^{p}
  }.
\end{equation*}
We claim
\begin{proposition}\label{thm:p3}
  $T(\Hp)$ is a cyclic group  of  order  $p-1$.
\end{proposition}

The arguments are similar to those for $\Gp$ in
Section~\ref{sec:Tgp}, so we only sketch them briefly. 

$\Aut(\Hp)$ 
induces  on the  $\F_{p}$-vector space  $V  = \Hp/\Frat(\Hp)$  the group  of
matrices
\begin{equation*}
  \Set{
  \begin{bmatrix}
    a & 0\\
    b & 1
  \end{bmatrix}
  :
  a \in \F_{p}^{*}, b \in \F_{p}
  }
\end{equation*}
with respect to the basis induced by $x, y$.
Clearly such an automorphism sends, using additive notation, $x^{p}$
to $a x^{p}$.

Arguing as in  Section~\ref{sec:Tgp}, one sees that  the
$\gamma$ for $\Hp$ are given by
\begin{equation*}
  \gamma_{t}(x)
  :
  \begin{cases}
    x \mapsto x\\
    y \mapsto x^{- p t} y
  \end{cases}
  \qquad
  \gamma_{t}(y)
  :
  \begin{cases}
    x \mapsto x^{1 + p t}\\
    y \mapsto y
  \end{cases},
\end{equation*}
for $t \in \F_{p}^{*}$.

For the regular
subgroup $N \norm
\Hol(G)$ corresponding to $\gamma_{t}$ we have
\begin{align*}
  [\nu(x), \nu(y^{d})]
  =
  \nu(x^{p d (1 + 2 t)}).
\end{align*}
Therefore $N$  is abelian when  $t = -1/2$.  (See Remark~\ref{rem:BCH}
below.) 

When $t \ne -1/2$, we choose $d =  (1 + 2 t)' \in \F_{p}^{*}$, so that
$\Hp \cong (\Hp, \circ)$, via the isomorphism defined by
\begin{equation*}
  \theta_{d} :
  \begin{cases}
    x \mapsto x\\
    y \mapsto y^{d}
  \end{cases}
\end{equation*}
and by the accompanying $\gamma_{t}$ as above.
Since we have  $p-1$ choices for $\theta_{d}$, it follows
from Proposition~\ref{prop:powers}  that $T(\Hp)$  is cyclic  of order
$p-1$.

\subsection{Groups of exponent $p$}

We claim the following
\begin{theorem}\label{thm:free}
  Let $G$ be the free $p$-group of class two and exponent $p > 2$ on $n \ge
  2$ generators.
  
  Then $T(G)$ is cyclic of order $p-1$.
\end{theorem}

Since $G$ is free in a variety, and $\Frat(G) = Z(G)$, we have that
the group $\Aut(G)/\Aut_{c}(G)$ is isomorphic to $ \GL(n, p)$.
The  conditions   in~\eqref{eq:gamma-for-normal}  imply  that
$\gamma(G)$ is a $p$-group, which is normal in $\Aut(G)$.
Therefore $\gamma(G) \le 
\Aut_{c}(G)$. Moreover Lemma~\ref{lemma:eqcond} implies 
\begin{equation*}
  [Z(G), \gamma(G)] = [G', \gamma(G)] = 1.
\end{equation*}
We    can     thus    use     the    the    $\Delta$     setting    of
Proposition~\ref{prop:Delta}. Here  $V = G/G'  = G/Z(G)$ and $W  = G'$
are  $\F_{p}$-vector spaces,  and $W$  is naturally  isomorphic to  the
exterior square $\exteriorpower^{2} V$ of $V$.

\subsubsection{The case $p > 3$}

Take first $0 \ne g  \in V$,   and    complete   it    to   a    basis   of $V$.

When $p > 3$, we can choose $a \in \F_{p}^{*} \setminus \Set{1, -1}$. Consider 
the  automorphism $\beta$ of $V$ (and thus of $G$),  that fixes  $g$,
and  multiplies all  other 
basis elements  by $a$. Then  $\beta$ fixes $\Delta(g, g)$,  but there
are no fixed points of $\beta$  in $W \cong \exteriorpower^{2} V$, as
$\beta$ multiplies all
natural basis elements of $\exteriorpower^{2} V$
by $a \ne 1$ or $a^{2} \ne 1$. Thus $\Delta(g, g) = 0$.

Take now two independent elements $g$  and $h$ of $V$, and complete
them to a basis of $V$. Consider  the automorphism $\beta$ of $V$
which fixes $g, h$, and multiplies  all other basis elements by $a \in
\F_{p}^{*} \setminus \Set{1, -1}$. Then the only fixed points of
$\beta$ in $W$ are the multiples 
of  $[g,  h]$, so  that  $\Delta(g,  h) =  c  [g,  h]$ for  some
$c$. 
Consider another basis elements  $k \ne g, h$, so
that $\Delta(g, k) = c' [g, k]$,  and $\Delta(g, h + k) = c [g,
  h] + c' [g,  k]$. Since $\Delta(g, h + k)$ must  also be a multiple
of $[g,  h +  k]$, we  see that  $c' =  c$ uniformly,  so that,
reverting to multiplicative notation,
\begin{equation*}
  g^{\gamma(h)}
  =
  g \Delta(g, h)
  =
  g [g, h]^{c}
  =
  g^{h^{c}}.
\end{equation*}
that is, $\gamma(h) = \iota(h^{c})$. For $c \ne -1/2$,
this is uniquely 
associated, as we have seen in Section~\ref{sec:powers}, to
$g^{\theta_{d}} = g^{d}$, where $d = (1 + 2 c)'$. 
\begin{remark}\label{rem:BCH}
  Similarly to what happened in Subsection~\ref{subsec:p3}, there is
  no such $\theta_{d}$ when $c = -1/2$.  
  In fact in this case Lemma~\ref{lemma:commutator-of-nu} shows that
  \begin{align*}
    [\nu(g), \nu(h)]
    &=
    \nu([g, h] + \Delta(g, h) - \Delta(h, g))
    \\&=
    \nu([g, h] - \frac{1}{2} [g, h] + \frac{1}{2} [h, g])
    \\&=
    \nu(0) =  1,
  \end{align*}
  that is, $N \cong (G, \circ)$ is abelian, and thus $N$ is not
  isomorphic to $G$. 

  Note that this is a particular case of the Baer
  correspondence~\cite{Baer-corr}, which is in turn an approximation of
  the Lazard correspondence and the  Baker-Campbell-Hausdorff
  formulas~\cite[Ch.~9~and 10]{khukhro}.
\end{remark}

\subsubsection{The case $p = 3$}

The  method  described in  the  following  works  for a  general  $p >
2$, 
although it is slightly more cumbersome than the previous one.

Let $0 \ne g \in V = G/G' = G/Z(G)$, and let $C$ be a complement to $\Span{g}$
in $G$. Consider the automorphism $\beta$ of $V$ which fixes
$g$, and acts as scalar multiplication by $-1$ on 
$C$. Then
$
  \Delta(g, g) \in \exteriorpower^{2} C.
$
Letting $C$ range  over all complements of $\Span{g}$,  we obtain that
$\Delta(g, g) = 0$.

Similarly, given two independent elements $g, h$ of $V$, let
$C$ be a complement to $\Span{g, h}$ in $V$. Define the automorphism
$\beta$ of 
$V$  which fixes  $g, h$, and acts as scalar multiplication by $-1$ on
$C$. Then  
\begin{equation*}
  \Delta(g, h) \in \Span{g \wedge h} + \exteriorpower^{2} C.
\end{equation*}
Letting $C$ range over all complements of
$\Span{g, h}$, we 
obtain that $\Delta(g, h) \in \Span{g \wedge h}$.

\subsection{A biggish $T(G)$}

For the next class of examples, we will use the $p$-groups
of~\cite{simple-con}. 
\begin{theorem}[\cite{simple-con}]\label{thm:all-central}
  Let $p > 2$ be a prime, and $n \ge 4$. 

  Consider the presentation
  \begin{equation}\label{eq:presentation}
    \begin{aligned}
      G
      =
      \left\langle x_{1}, \dots, x_{n}\right.
      :\
      &\text{$[[x_{i}, x_{j}], x_{k}] = 1$ for all $i, j, k$,}
      \\&\text{$x_{i}^{p} = \prod_{j < k} [x_{j}, x_{k}]^{a_{i,j,k}}$
        for all $i$,}
      \left. \right\rangle,
    \end{aligned}
  \end{equation}
  where $a_{i,j,k} \in \F_{p}$.

  There  is   a  choice  of  the   $a_{i,j,k}  \in  \F_{p}$
  such  that  the following hold.
  \begin{itemize}
  \item $G$ has nilpotence class two,
  \item $G$ has order $p^{n + \binom{n}{2}}$, 
  \item $G/\Frat(G)$ has order $p^{n}$,
  \item $G' = \Frat(G) = Z(G)$ has order $p^{\binom{n}{2}}$ and
    exponent $p$,
  \item $\Aut(G) =  \Aut_{c}(G)$, that
  is, all of the automorphism of $G$ are central.
  \end{itemize}
\end{theorem}

We claim the following
\begin{theorem}\label{thm:big}
  Let $G$ be one of the groups of Theorem~\ref{thm:all-central}.

  Then $T(G)$ contains a non-abelian subgroup of order 
  \begin{equation*}
    (p-1) \cdot p^{\binom{n}{2} \binom{n+1}{2}},
  \end{equation*}
  which  is the extension
  of an $\F_{p}$-vector space of dimension $\binom{n}{2} \binom{n+1}{2}$ by the
  multiplicative group $\F_{p}^{*}$ acting naturally.
\end{theorem}

Here again $V = G/G' = G/Z(G)$ and $W  = G'$ can be regarded as $\F_{p}$-vector
spaces,  and  $W$ is  naturally  isomorphic  to $\exteriorpower^{2}  V$.  The
conditions of Proposition~\ref{prop:Delta} apply, so we can use the $\Delta$
setting. Because of Theorem~\ref{thm:all-central},
condition~\eqref{eq:Delta-beta} of Proposition~\ref{prop:Delta} holds
trivially, so that we are simply looking  at all bilinear maps $\Delta : V
\times V \to W$ here.

Consider all such $\Delta$ which are symmetric with respect to the basis of
$V$ induced by the $x_{i}$. Under the pointwise operation of $W$, these
$\Delta$ form  a 
vector space 
$\mathfrak{D}$ of dimension 
$\binom{n}{2} \binom{n+1}{2}$ over $\F_{p}$. 
For the regular subgroup $N \norm \Hol(G)$ of $S(G)$ corresponding to
such a $\Delta$,
Lemma~\ref{lemma:commutator-of-nu} yields, for $1 \le i < j \le n$,
\begin{equation}\label{eq:comm-of-nu}
  [\nu(x_{i}), \nu(x_{j})] 
  =
  \nu([x_{i}, x_{j}] + \Delta(x_{i}, x_{j}) - \Delta(x_{j}, x_{i}))
  =
  \nu([x_{i}, x_{j}])
\end{equation}
and Lemma~\ref{lemma:powers} yields, for $1 \le i \le n$,
\begin{align*}
  \nu(x_{i})^{p} 
  = 
  \nu(x_{i}^{p} \cdot \Delta(x_{i}, x_{i})^{\binom{p}{2}})
  =
  \nu(x_{i}^{p}),
\end{align*}
since $\Delta(x_{i}, x_{i}) \in Z(G)$, and the latter group has
exponent $p$.
Thus in the corresponding group $(G, \circ) \cong N$ commutators and
$p$-th powers of generators are 
preserved, so that in view of~\eqref{eq:presentation}
for each $\Delta \in \mathfrak{D}$ there is an isomorphism $\theta_{\Delta} : G
\to (G, \circ)$ such that $x_{i} \mapsto x_{i}$ for each $i$. Since
the composition of $\theta_{\Delta}$ followed by $\nu$ is an
isomorphism $G \to N$, we obtain from~\eqref{eq:comm-of-nu} 
\begin{equation*}
  \nu([x_{i}, x_{j}]^{\theta_{\Delta}})
  =
  [\nu(x_{i}^{\theta_{\Delta}}), \nu(x_{j}^{\theta_{\Delta}})]
  =
  [\nu(x_{i}), \nu(x_{j})]
  =
  \nu([x_{i}, x_{j}]),
\end{equation*}
that is,
$\theta_{\Delta}$ fixes the elements of $G'$.

If $\Delta_{1}, \Delta_{2} \in \mathfrak{D}$, then for $1 \le i, j \le
n$ we have, since the $\Delta$ take values in $G'$,
\begin{align*}
  (x_{i} x_{j})^{\theta_{\Delta_{1}} \theta_{\Delta_{2}}} 
  &=
  (x_{i} x_{j} \Delta_{1}(x_{i}, x_{j}))^{\theta_{\Delta_{2}}}
  \\&=
  x_{i} x_{j} \Delta_{2}(x_{i}, x_{j})  \Delta_{1}(x_{i}, x_{j}),
\end{align*}
where we have applied
Lemma~\ref{lemma:from-theta-to-gamma}\eqref{eq:theta-gamma}, using the
facts that $\theta_{\Delta}$ fixes the elements of $G' = Z(G)$, and
$\gamma(G') = 1$. 

Therefore $\Set{ \theta_{\Delta} : \Delta \in \mathfrak{D} } \cong
\mathfrak{D}$ is an 
elementary abelian group of order $p^{\binom{n}{2}
  \binom{n+1}{2}}$. Considering also 
the $\theta_{d}$ of Section~\ref{sec:powers}, we readily obtain
Theorem~\ref{thm:big}. 

\providecommand{\bysame}{\leavevmode\hbox to3em{\hrulefill}\thinspace}
\providecommand{\MR}{\relax\ifhmode\unskip\space\fi MR }
\providecommand{\MRhref}[2]{%
  \href{http://www.ams.org/mathscinet-getitem?mr=#1}{#2}
}
\providecommand{\href}[2]{#2}

\end{document}